\documentclass[10pt]{amsart}
\usepackage[utf8]{inputenc}
\usepackage{graphicx}
\usepackage{float}
\usepackage{amsmath}
\usepackage{amssymb}
\usepackage{enumerate}
\usepackage{verbatim}
\usepackage{tcolorbox}
\usepackage[font=footnotesize]{caption}
\usepackage{url}
\usepackage{textcomp}
\usepackage{color}
\usepackage{hyperref}
\usepackage{multicol}
\usepackage{wrapfig}
\usepackage{gensymb}
\usepackage{subcaption}
\usepackage{transparent}
\usepackage{amsthm}
\usepackage{mathtools}
\usepackage{enumitem}
\usepackage{tikz-cd}
\DeclareMathOperator{\interior}{Int}

\allowdisplaybreaks

\newtheorem{lemma}{Lemma}[section]
\newtheorem{theorem}[lemma]{Theorem}
\newtheorem{definition}[lemma]{Definition}
\newtheorem{prop}[lemma]{Proposition}

\newtheorem{coro}[lemma]{Corollary}

\newtheorem{ex}{Example}[section]

\newtheorem{mtheorem} {Theorem}

\newcommand{\rmd}{\mathrm{d}}

\title[Smooth persistence of attractors for set-valued dynamical systems]{Smooth persistence of attractors\\ for set-valued dynamical systems:\\a boundary map approach}

\author{K. Kourliouros}
\address{Department of Mathematics, Imperial College London, 180 Queen's Gate, London SW7 2AZ, United Kingdom}
\email{kk110@imperial.ac.uk}
\thanks{This research has been supported by the EPSRC grants EP/W009455/1 (KK, JSWL, MR) and  EP/Y020669/1 (KK, JSWL, MR, WHT), as well as EU Marie Sklodowska-Curie ITN Critical Transitions in Complex Systems (H2020-MSCA-ITN-2014 643073 CRITICS) (JSWL, MR, KGT).
JSWL also acknowledges support from the EPSRC Centre for Doctoral Training in Mathematics of Random Systems: Analysis, Modelling and Simulation (EP/S023925/1) and thanks IRCN (Tokyo) for their support. JSWL and WHT are grateful for support from JST Moonshot R \& D Grant Number JPMJMS2021, and  WHT was supported also by EPSRC PhD scholarship grant EP/S515085/1 and the Project of Intelligent Mobility Society Design, Social Cooperation Program, UTokyo. KGT has been supported by the European Research Council Advanced Grant (ERC AdG, ILLUSIVE: Foundations of Perception Engineering, 101020977). DT acknowledges support from the Leverhulme Trust (RPG-2021-072).}

\author{J.S.W. Lamb}
\address{Department of Mathematics, Imperial College London, 180 Queen's Gate, London SW7 2AZ, United Kingdom}
\address{International Research Center for Neurointelligence (IRCN), The University of Tokyo, 7-3-1 Hongo Bunkyo-Ku, Tokyo, 113-0033, Japan}
\email{jsw.lamb@imperial.ac.uk}
\thanks{}

\author{M. Rasmussen}
\address{Department of Mathematics, Imperial College London, 180 Queen's Gate, London SW7 2AZ, United Kingdom}
\email{m.rasmussen@imperial.ac.uk}
\thanks{}

\author{W. H. Tey}
\address{Department of Mathematics, Imperial College London, 180 Queen's Gate, London SW7 2AZ, United Kingdom}
\email{w.tey18@imperial.ac.uk}
\thanks{}

\author{K. G. Timperi}
\address{Center for Ubiquitous Computing, University of Oulu,
FI-90014 Oulu, Finland}
\email{kalle.timperi@oulu.fi}
\thanks{}

\author{D. Turaev}
\address{Department of Mathematics, Imperial College London, 180 Queen's Gate, London SW7 2AZ, United Kingdom}
\email{d.turaev@imperial.ac.uk}
\thanks{}

\begin{document}

\begin{abstract}
We study the problem of persistence of attractors with smooth boundary for a class of set-valued dynamical systems that naturally arise in the context of random and control dynamical systems, as well as in systems modeling the dynamical propagation of uncertainty. In order to tackle the inherent difficulties associated to the multi-valued structure of such dynamical systems, we introduce a single-valued map, the so-called boundary map, which is a contactomorphism of the unit-tangent bundle of the state space, with the following characteristic property: boundaries of attractors of the set-valued dynamical system correspond in a unique way to invariant Legendrian manifolds of this map. We show how the underlying contact geometry guarantees the smooth persistence of such attractors under perturbations of the set-valued dynamical system, provided that the associated boundary map is normally hyperbolic at the unit normal bundle of the boundary.  
\end{abstract}


\maketitle

\section{Introduction-Main Results}

Let $f:\mathbb{R}^d\rightarrow \mathbb{R}^d$ be a $C^r$-diffeomorphism, $r\geq 1$, $d\geq 2$, whose iterations generate a discrete-time dynamical system. For a given perturbation parameter $\varepsilon>0$, we consider an associated set-valued dynamical system on the space $\mathcal{K}(\mathbb{R}^d)$ of all compact subsets of $\mathbb{R}^d$, obtained by iterating the map $F_{f,\varepsilon}:\mathcal{K}(\mathbb{R}^d)\rightarrow \mathcal{K}(\mathbb{R}^d)$, defined by
\begin{equation}\label{eq:sv-dyn}
  F_{f,\varepsilon} (A) := \overline{B_\varepsilon(f(A))}:=\bigcup_{x\in f(A)}\overline{B_{\varepsilon}(x)}\quad \mbox{for all } A\in\mathcal K(\mathbb R^d)\,,
\end{equation} 
where $\overline{B_{\varepsilon}(x)}$ denotes the closed Euclidean $\varepsilon$-ball centered at $x$. 


An invariant set $M\in\mathcal K(\mathbb R^d)$ is defined as a fixed point of the set-valued map, namely: 
\begin{displaymath}
  F_{f,\varepsilon}(M)=M,
\end{displaymath} 
and is called minimal, if there does not exist a proper subset $M'\subset M$ which is also invariant for $F_{f,\varepsilon}$. A (minimal) attractor is a (minimal) invariant set which attracts a neighbourhood of itself with respect to the Hausdorff metric $d_H$, in the sense that there exists $\delta>0$ such that
\begin{displaymath}
    \lim_{k\to\infty} d_H\big (F^k_{f,\varepsilon}(B_\delta(A)), A\big) = 0 \,.
\end{displaymath}

 Set-valued dynamical systems of the form (\ref{eq:sv-dyn}) and their invariant sets, appear either explicitly or implicitly in various contexts, where they usually model the compound behaviour of all possible $\varepsilon$-pseudo-orbits of the diffeomorphism $f$, given in set-valued terms by a dynamical system (a difference inclusion) of the form:
\[x_{n+1}\in F_{f,\varepsilon}(x_n):=\overline{B_{\varepsilon}(x_n)}, \quad n\in \mathbb{N}.\] 
For instance, in the theory of $\varepsilon$-pseudo-orbits and the associated theory of shadowing cf.~\cite{B1,B2,B3}, the minimal invariant sets of $F_{f,\varepsilon}$ represent the so-called $\varepsilon$-chain transitive sets, and they play a central role in Conley's decomposition theory \cite{Con1,Con2,E1} (see also ~\cite{MCG,MCGW,MCG1} for generalisations to arbitrary closed relations). In control theory, minimal invariant sets of $F_{f,\varepsilon}$ appear as so-called control sets \cite{BO,CoKl}, i.e. those subsets of the state space where maximal approximate controlability holds and no orbit can escape by any choice of admissible controls in the closed $\varepsilon$-ball. Most notably, in the theory of random dynamical systems, minimal invariant sets of $F_{f,\varepsilon}$ correspond to the supports of stationary ergodic measures for the random dynamical system generated by $f$ with additive i.i.d noise from the closed $\varepsilon$-ball ~\cite{Ara,BBS,LRR,MMCG,ZH}. 


Motivated mostly by the above applications, several forms of topological stability, persistence and bifurcations of invariant sets of set-valued dynamical systems have appeared in the literature (cf.~\cite{Gay1,Gay2,Guz1,HY,HYG,HY1,LRR,Li1,MCG,ZH}). 
In particular, it has been recently shown in \cite{LRTeyB} that bifurcation scenaria in set-valued systems of the form \eqref{eq:sv-dyn}, involve both the sudden changes in the topology of the invariant sets (so called ``topological bifurcations" \cite{LRR}), as well as changes in the smoothness properties of their boundaries (so called ``boundary bifurcations"), such as for example loss of regularity and/or creation of singularities. 

On the other hand, there is an open regime in the parameter space where such boundary bifurcations of minimal invariant sets do not occur at all, i.e.  ~the differentiable (regularity) type of the boundary remains constant under small perturbations of the set-valued dynamical system. We call this phenomenon \textit{smooth persistence}. Our main result in this paper (Theorem \ref{thm:main2} below) is to give sufficient conditions which guarantee the smooth persistence of minimal invariant sets, and in particular, the absence of boundary bifurcations.


The above problems of smooth persistence and boundary bifurcations of invariant sets for set-valued dynamical systems are challenging problems, as the space $\mathcal K(\mathbb R^d)$ of all compact subsets of $\mathbb{R}^d$ does not have a Banach space structure, so that classical tools from differentiable dynamical systems theory (like the implicit function theorem) are absent in this setting. This  creates obvious obstacles for the development of the theory as well as for practical numerical algorithms \cite{CoKl}. 

In order to overcome this challenge, we introduce in this article a novel approach, having its origins in geometric control theory and the well-known Pontryagin maximum principle (cf.~\cite{AS,BO,Gam}) in a discrete setting. In particular, to any set-valued map $F_{f,\varepsilon}$ of the form (\ref{eq:sv-dyn}), we associate a single-valued self-map $\beta_{f,\varepsilon}$ of the unit tangent bundle $T_1\mathbb{R}^d=\mathbb{R}^d\times S^{d-1}$ of $\mathbb R^d$, which we call the \textit{boundary map}. It has the characteristic property that it maps points $(x,n)$ of $T_1\mathbb{R}^d$ to points on the outward unit normal bundle $N_1^+\partial F_{f,\varepsilon}(x)$ of the boundary of the $\varepsilon$-ball around $f(x)$,
 and is given by 
\begin{displaymath}
  \beta_{f,\varepsilon}(x,n)=\left (f(x)+\varepsilon \frac{f'(x)^{-T}n}{\|f'(x)^{-T}n\|}, \frac{f'(x)^{-T}n}{\|f'(x)^{-T}n\|}\right ) \quad \mbox{for all } (x,n)\in T_1\mathbb{R}^d
\end{displaymath} 
 (see also Definition~\ref{def-0} below).
 
The boundary map has many useful differential geometric properties inherent in its definition, the most important being the preservation of the natural contact structure of the unit tangent bundle $T_1\mathbb{R}^d$. Thus, the boundary map $\beta_{f,\varepsilon}$ defines a $C^{r-1}$-contactomorphism of the unit tangent bundle, and as a consequence, it preserves the class of Legendrian submanifolds as well, i.e.~~the class of those submanifolds which are maximally integrable for the contact distribution. It is exactly this property that allows us to establish a bijection between the invariant sets $M$ of the set-valued map $F_{f,\varepsilon}$ with smooth boundaries $\partial M$, and a special class of invariant Legendrian submanifolds of the boundary map $\beta_{f,\varepsilon}$, consisting of the outer unit normal bundles $N_1^+\partial M$ over smooth closed hypersurfaces $\partial M$ in $\mathbb{R}^d$.

\begin{mtheorem}
\label{theorem:invariant0}
Let $M\in \mathcal{K}(\mathbb{R}^d)$ be a compact set with $C^r$-smooth boundary $\partial M$, $r\geq 1$. Then $M$ is invariant for the set-valued map $F_{f,\varepsilon}$ if and only if its outer unit normal bundle $N_1^+\partial M$ is invariant for the boundary map $\beta_{f,\varepsilon}$.
\end{mtheorem}

The theorem is reminiscent of Huygen's principle in wave front evolution (cf.~\cite{A1,A2}) in a discrete time setting. Its proof is given in Section \ref{sec:2}, Theorem \ref{theorem:invariant}. It ultimately relies on the contact geometry of equidistant hypersurfaces inherent in the construction of the boundary map $\beta_{f,\varepsilon}$, which will be analysed in detail in the same section. 

Our next main result concerns the aforementioned problem of smooth persistence of minimal invariants sets.   
Crucially this can be done due to the boundary map, which enables the employment of the classical notion of normal hyperbolicity (cf.~\cite{Fen1,HPS}) leading to smooth persistence results in the set-valued setting. We show in particular that the presence of the contact structure enforces such minimal invariants sets to be minimal attractors.

\begin{mtheorem}
\label{thm:main2}
  Let $f$ be a $C^2$-diffeomorphism and let $M\in \mathcal{K}(\mathbb{R}^d)$ be a minimal invariant set for the set-valued map $F_{f,\varepsilon}$, such that its boundary $\partial M$ is a $C^2$-smooth closed hypersurface. Suppose that the boundary map $\beta_{f,\varepsilon}$ is normally hyperbolic at the outer unit normal bundle $N_1^+\partial M$ of the boundary $\partial M$. Then $M$ is a minimal attractor, and it is $C^2$-persistent for the set-valued map $F_{f,\varepsilon}$, i.e.~for any diffeomorphism $\tilde{f}$ which is $C^2$-close to $f$ and any $\tilde{\varepsilon}$ sufficiently close to $\varepsilon$, there exists a unique minimal attractor $\tilde{M}$ for the perturbed set-valued map $F_{\tilde{f},\tilde{\varepsilon}}$, such that its boundary $\partial \tilde{M}$ is diffeomorphic and $C^2$-close to $\partial M$.
\end{mtheorem}

The proof of the Theorem is presented in Section \ref{sec:3} (Theorem \ref{thm:persistence}) and relies on the contact properties of the boundary map. It consists of establishing three facts: (a) normal hyperbolicity of the boundary map $\beta_{f,\varepsilon}$ at the unit normal bundle $N_1^+\partial M$ of the boundary of a minimal invariant set $M$ for $F_{f,\varepsilon}$, implies that $M$ is a minimal attractor, (b) the existence of a unique invariant manifold for the perturbed boundary map $\beta_{\tilde{f},\tilde{\varepsilon}}$, which is the outer unit normal bundle $N_1^+\partial \tilde{M}$ over a smooth closed hypersurface $\partial \tilde{M}$, diffeomorphic and $C^2$-close to $\partial M$, and (c) minimality of the invariant set $\tilde{M}$ for the perturbed set-valued map $F_{\tilde{f},\tilde{\varepsilon}}$. The proof of (a) is a consequence of a general fact in contact (and symplectic) geometry and dynamics, according to which any contactomorphism of a contact manifold which is normally hyperbolic at an invariant Legendrian submanifold, is either normally attracting or normally repelling, i.e. normal saddles are prohibited in the presence of a contact structure (see Proposition \ref{prp:attractive-repelling}). The proof of (b) follows from a contact version of the graph transform method, tailored for the purposes of the boundary map. Finally, the proof of (c) is topological and relies on the fact that minimality of invariant sets persists under small deformations of the set-valued map.  

We note that the assumption of $C^2$-smoothness of the boundary enables the use of normal hyperbolicity arguments for establishing smooth persistence. However, we conjecture that the assumption of $C^2$-smoothness of the boundaries of minimal invariant sets is a necessary condition for their smooth persistence.  Indeed, numerical studies in \cite{LRTeyB} show that if the boundary of a minimal invariant set is merely $C^1$ but not $C^2$-smooth, then $C^1$-smooth persistence fails due to the possibility of the creation of wedge singularities (self-intersection of the boundary) by arbitrarily small perturbations.   
On the one hand, in analogy with results of Man\'e \cite{Man}, we conjecture that normal hyperbolicity of the boundary map is also a necessary condition for the $C^2$-persistence of minimal invariant sets $M$.

We illustrate Theorem~\ref{thm:main2} with an elementary example:
\begin{ex}
\label{ex:hc}
Consider the linear map $A(x)=\lambda x$ where $|\lambda|<1$ and $x\in\mathbb{R}^d$. Then, for any $\varepsilon >0$, there exists a $C^2$-neighbourhood of $A$ such that for any $C^2$-diffeomorphism $f:\mathbb{R}^d\to\mathbb{R}^d$ in this neighbourhood, the set-valued map $F_{f,\varepsilon}$ has a unique attractor $M$ that is diffeomorphic and $C^2$-close to the ball $\overline{B_{\frac{\varepsilon}{1-|\lambda|}}(0)}$. 
\end{ex}
\begin{proof}
It is readily verified that for any $\varepsilon>0$ the closed ball $\overline{B_{\frac{\varepsilon}{1-|\lambda|}}(0)}$  is the unique attractor for the set-valued map $F_{A,\varepsilon}$. Moreover, the outward unit normal bundle $N_1^+S^{d-1}_{\frac{\varepsilon}{1-|\lambda|}}(0)$ of the boundary of this ball is pointwise fixed by the boundary map $\beta_{A,\varepsilon}$ and the eigenvalues of its derivative in normal directions equal $\lambda$, implying normal attraction.
Hence, for any $f$ sufficiently $C^2$-close to $A$, the associated boundary map (with the same $\varepsilon$) $\beta_{f,\varepsilon}$ is $C^1$-close to $\beta_{A,\varepsilon}$, and the result follows by Theorem \ref{thm:main2}.   
\end{proof}

In relation to the results in \cite{LRTey}, we conjecture that Example \ref{ex:hc} above can be generalised to the generic case, i.e. for an open and dense subset of the space of invertible linear maps $L\in GL(d,\mathbb{R})$ with spectral radius smaller than one, and satisfying certain spectral gap conditions which guarantee both the $C^2$-smoothness of the boundaries of minimal invariant sets, as well as the normal hyperbolicity of the boundary map at their unit normal bundles.



Finally, we would like to remark that the applicability of the boundary map ranges much beyond the results of the present article. For example, it has already been used  in the (numerical) investigation of bifurcations of attractors \cite{LRTeyB}. Also, the boundary map can be used to study the dynamics (ie time-evolution) of sets, and not only the attractors.
Moreover, it can be generalised under appropriate modifications, to more general set-valued dynamical systems, both of discrete and continuous time, as those arising from general difference and differential inclusions, beyond the hypothesis of closed $\varepsilon$-balls, or when the state space is a smooth manifold instead of the flat Euclidean space.   

\section{The boundary map}
\label{sec:2}

We denote by $\mathcal{K}(\mathbb{R}^d)$ the set of all compact subsets of $\mathbb{R}^d$. This is a complete metric space when endowed with the Hausdorff distance $d_H$, which is defined as follows: for any two compact sets $A, B\in \mathcal K (\mathbb{R}^d)$, their Hausdorff distance is the maximum of their Hausdorff semi-distances, $d_H(A,B):=\max \big \{d(A,B),d(B,A)\big \}$
where $d(A,B):=\sup_{x\in A}d(x,B):=\sup_{x\in A}\inf_{y\in B} \|x-y\|$ and $\|\cdot\|$ is the Euclidean norm on $\mathbb R^d$.

Let $\varepsilon>0$ and $f:\mathbb{R}^d\rightarrow \mathbb{R}^d$, $d\geq 2$, be a $C^r$-diffeomorphism, $r\geq 1$. Then the set-valued map associated to $(f,\varepsilon)$ is the map 
\[F_{f,\varepsilon}:\mathbb{R}^d\rightarrow \mathcal{K}(\mathbb{R}^d),\]
defined by:
\[F_{f,\varepsilon}(x):=\overline{B_{\varepsilon}(f(x))}:=\big \{y\in \mathbb{R}^d: \|y-f(x)\|\leq \varepsilon \big \}.\]
It naturally lifts to a map $F_{f,\varepsilon}:\mathcal{K}(\mathbb{R}^d)\rightarrow \mathcal{K}(\mathbb{R}^d)$ by setting, for any compact $A\subset \mathbb{R}^d$:
\[F_{f,\varepsilon}(A)=\bigcup_{x\in A}F_{f,\varepsilon}(x).\]

A compact set $M\subset \mathbb{R}^d$ will be called forward (resp. backward) invariant for $F_{f,\varepsilon}$, if $F_{f,\varepsilon}(M)\subseteq M$ (resp. $M\subseteq F_{f,\varepsilon}(M)$), and it will be called invariant for $F_{f,\varepsilon}$ if it is both forward and backward invariant, i.e.:
\[F_{f,\varepsilon}(M)=M.\]
An invariant set $M$ will be called \emph{minimal invariant} if there does not exist any proper subset of $M$ which is also invariant for $F_{f,\varepsilon}$. 

Notice now that if $M\subset \mathbb{R}^d$ is a compact invariant set for $F_{f,\varepsilon}$, then its boundary $\partial M$ obviously satisfies
\[\partial (F_{f,\varepsilon}(M))=\partial M,\]
and conversely. On the other hand, the boundary $\partial M$ itself is only backward invariant for $F_{f,\varepsilon}$, i.e.:
\[\partial M\subset F_{f,\varepsilon}(\partial M).\]
In particular:
\[\partial M=\partial (F_{f,\varepsilon}(M))=\partial_+(F_{f,\varepsilon}(\partial M)):=\partial (F_{f,\varepsilon}(\partial M))\cap (\overline{\mathbb{R}^d\setminus M})\]
i.e.~$\partial M$ can be identified with the \emph{outer boundary} of $F_{f,\varepsilon}(\partial M)$, denoted by $\partial_+(F_{f,\varepsilon}(\partial M))$ above, which is the set of connected components that lie in the closure of the complement $\mathbb{R}^d\setminus M$. 

Given now a set-valued map $F_{f,\varepsilon}$ as above associated to a pair $(f,\varepsilon)$, we introduce the so-called boundary map $\beta_{f,\varepsilon}$ of $F_{f,\varepsilon}$, which is a single-valued map defined on the unit tangent bundle  $T_1\mathbb{R}^d=\mathbb{R}^d\times S^{d-1}$. In particular:

\begin{definition}[Boundary map]\label{def-0}
  Let $F_{f,\varepsilon}:\mathcal{K}(\mathbb{R}^d)\rightarrow \mathcal{K}(\mathbb{R}^d)$ be the set-valued map associated to $(f,\varepsilon)$, where $f:\mathbb{R}^d\rightarrow \mathbb{R}^d$ is a diffeomorphism and $\varepsilon>0$. Then the boundary map of $F_{f,\varepsilon}$ is given by $\beta_{f,\varepsilon}:T_1\mathbb{R}^d\rightarrow T_1\mathbb{R}^d$, where
  \[\beta_{f,\varepsilon}(x,n)=\left (f(x)+\varepsilon \frac{f'(x)^{-T}n}{\|f'(x)^{-T}n\|}, \frac{f'(x)^{-T}n}{\|f'(x)^{-T}n\|}\right ) \quad \mbox{for all } (x,n)\in T_1\mathbb{R}^d\,.\] 
\end{definition}

We describe several immediate geometric properties of the boundary map $\beta_{f,\varepsilon}$. Firstly, we note that it can be expressed as a composition of two maps on the unit tangent bundle, $\beta_{f,\varepsilon} = \phi_{\varepsilon}\circ h_f$, where 
\[h_f:T_1\mathbb{R}^d\rightarrow T_1\mathbb{R}^d\,, \qquad h_f(x,n)=\left (f(x),\frac{f'(x)^{-T}n}{\|f'(x)^{-T}n\|} \right ),\]
and
\[\phi_{\varepsilon}:T_1\mathbb{R}^d\rightarrow T_1\mathbb{R}^d\,, \qquad \phi_{\varepsilon}(x,n)=\left (x+\varepsilon n,n\right ).\]
The map $h_f$ is the standard expression for the lift of the diffeomorphism $f$ on the unit cotangent bundle $T^*_1\mathbb{R}^d$, after identifying it with the unit tangent bundle using the Euclidean metric. 
The map $\phi_{\varepsilon}$ on the other hand, is the time-$\varepsilon$ map of the geodesic flow of $\mathbb{R}^d$, and its first component is the so-called \emph{exponential map}
\[\operatorname{Exp}_{\varepsilon}:=\pi \circ \phi_{\varepsilon}:T_1\mathbb{R}^d\rightarrow \mathbb{R}^d\,, \qquad \operatorname{Exp}_{\varepsilon}(x,n)=x+\varepsilon n\,,\]  
where $\pi:T_1\mathbb{R}^d\rightarrow \mathbb{R}^d$ is the unit tangent bundle projection. It follows that the boundary map $\beta_{f,\varepsilon}=\phi_{\varepsilon}\circ h_f$ is a $C^{r-1}$-diffeomorphism of the unit tangent bundle (as a composition of two diffeomorphisms). In fact more is true. Recall that the unit tangent bundle $T_1\mathbb{R}^d$ has a natural contact structure $\Delta_a\subset T(T_1\mathbb{R}^d)$, i.e.~a maximally non-integrable distribution of hyperplanes (codimension 1), obtained by the natural contact structure of the unit cotangent bundle $T^*_1\mathbb{R}^d$, after identifying vectors and covectors using the Euclidean metric. In coordinates $(x,n)$ of $T_1\mathbb{R}^d$, it is expressed as the field of kernels of the so-called Liouville 1-form:
\[a=n\rmd x:=\sum_{i=1}^dn_i\rmd x_i\,, \qquad \sum_{i=1}^dn_i^2=1,\]
\[\Delta_a=
\ker a,\]
where maximal non-integrability of this distribution is encoded in the relation $a\wedge (da)^{d-1}\ne 0$.

\begin{prop}
\label{prop-0}
The boundary map $E_{f,\epsilon}$ preserves the contact structure $\Delta_a\subset T(T_1\mathbb{R}^d)$ of the unit tangent bundle.
\end{prop}
\begin{proof}
It suffices to show that equation $a=0$ implies $\beta_{f,\varepsilon}^*a=0$. But $\beta_{f,\varepsilon}=\phi_{\varepsilon}\circ h_f$ and thus $\beta_{f,\varepsilon}^*a=h_f^*(\phi_{\varepsilon}^*a)$. The geodesic flow preserves the contact structure of the unit tangent bundle:
\[\phi_{\varepsilon}^*a=n\rmd x+\varepsilon n\rmd n=n\rmd x=a,\]
(because $n\rmd n=0$ whenever $\|n\|=1$), so it suffices to show that $h_f^*a=0$ whenever $a=0$. A simple calculation now shows that 
\[h_f^*a=\frac{1}{\|f'(x)^{-T}n\|}a\,,\]
which proves the result. 
\end{proof}

Diffeomorphisms which preserve the contact structure of a contact manifold are usually called \textit{contactomorphisms}. In this terminology, Proposition \ref{prop-0} above says that the boundary map $\beta_{f,\varepsilon}$ is a $C^{r-1}$-contactomorphism of the unit tangent bundle.

As is well known, among all submanifolds of a contact manifold of dimension $2d-1$ there is a distinguished class of $(d-1)$-dimensional submanifolds, called \textit{Legendrian}, which are characterised by the property of being the maximal (with respect to dimension) integral submanifolds of the contact structure. It is obvious from the definition that any contactomorphism sends Legendrian submanifolds to Legendrian submanifolds, and from this we obtain: 
  
\begin{coro}
\label{coro-0}
The boundary map $\beta_{f,\varepsilon}$ sends Legendrian submanifolds to Legendrian submanifolds of the unit tangent bundle.
\end{coro}

Within the set of Legendrian submanifolds of the unit tangent bundle there is again a distinguished class of submanifolds, consisting of unit normal bundles of submanifolds (of arbitrary dimension) of the base $\mathbb{R}^d$; if $\Gamma \subset \mathbb{R}^d$ is such a submanifold, then its unit normal bundle is a fiber bundle defined as:
\[N_1\Gamma:=\bigcup_{x\in \Gamma}N_{1,x}\Gamma \rightarrow \Gamma,\]
where 
\[N_{1,x}\Gamma:=\big \{n\in \mathbb{R}^d: \langle n,u\rangle=0\quad \forall u\in T_x\Gamma, \quad \|n\|=1\big \},\]
is the space of unit normal vectors at each point $x\in \Gamma$. Note that if $\Gamma$ is a smooth closed hypersurface in $\mathbb{R}^d$ (i.e.~a compact submanifold without boundary of codimension 1), then it is orientable and so its unit normal bundle consists of two connected components:
\[N_1\Gamma=N_1^+\Gamma \bigcup N_1^-\Gamma,\]
given by the outward (with the `$+$'-sign) or the inward (with the `$-$'-sign) unit vectors respectively. It is immediate then to show that each one of these components is a Legendrian submanifold of $T_1\mathbb{R}^d$ naturally diffeomorphic to $\Gamma$:
\[N_1^{\pm}\Gamma \cong \Gamma,\]
the diffeomorphism being provided by the restriction of the projection map $\pi:T_1\mathbb{R}^d\rightarrow \mathbb{R}^d$ on $N_1^{\pm}\Gamma$.

In general, if $L\subset T_1\mathbb{R}^d$ is a Legendrian submanifold, then its projection to the base $\mathbb{R}^d$ defines a map:
\[\pi|_L:L\rightarrow \mathbb{R}^d,\]
usually called a \textit{Legendrian map}. Its image defines (generically) a hypersurface
$\Gamma=\pi(L)$ of $\mathbb{R}^d$, usually called the \textit{(wave) front} of $L$. Typically, the wave front $\Gamma$ of a Legendrian submanifold $L$ has singularities cf.~\cite{A1,A2}, but the following lemma says that in case it is a smooth hypersurface, the Legendrian manifold $L$ is nothing but the unit normal bundle of its wave front. 

\begin{lemma}
\label{lem:front}
Suppose that $L$ is a Legendrian submanifold of the unit tangent bundle $T_1\mathbb{R}^d$ such that its front $\pi(L)=\Gamma$ is a smooth hypersurface of $\mathbb{R}^d$. Then $L$ is the unit normal bundle of $\Gamma$:
\[L=N_1\Gamma.\]
\end{lemma}
\begin{proof}
The proof is a tautology; namely, it suffices to show that for any point $(x,n)\in L$ and any tangent vector $v\in T_x\Gamma$, the following orthogonality condition $\langle n,v \rangle=0$ holds, i.e.~that $n\in N_{1,x}\Gamma$. But since the map $\pi|_L$ is a diffeomorphism onto $\Gamma$, its differential $d\pi|_L:TL\rightarrow T\mathbb{R}^d$ is an isomorphism onto $T\Gamma$, and thus for any $v\in T_x\Gamma$ there exists a unique $\xi \in T_{(x,n)}L$ such that $v=d\pi_{(x,n)}(\xi)$. Thus
 \[\langle v,n \rangle= \langle d \pi_{(x,n)}(\xi),n\rangle =a_{(x,n)}(\xi)=0,\]
 where the second equality follows by definition of the contact structure and the last equality follows by the fact that $L$ is Legendrian.
\end{proof}

Let us now see what is the effect of the application of the boundary map $\beta_{f,\varepsilon}$ on the outer unit normal bundle of a smooth closed hypersurface $\Gamma$ of $\mathbb{R}^d$. Recall that by definition, given such a hypersurface $\Gamma\subset \mathbb{R}^d$ its (outer) {\it normal $\varepsilon$-equidistant} (or outer $\varepsilon$-parallel) is the hypersurface  
\[\Gamma_{\varepsilon}:=\operatorname{Exp}_{\varepsilon}(N_1^+\Gamma)=\big \{y\in \mathbb{R}^d: y=x+\varepsilon n,\hspace{0,2cm} (x,n)\in N_1^+\Gamma\big \}\]
formed by the locus of points whose normal distance from $\Gamma$ is equal to $\varepsilon$. From this it follows:
\[\partial_+ \overline{B_{\varepsilon}(\Gamma)}\subseteq \Gamma_{\varepsilon},\]
with equality if $\Gamma_{\varepsilon}$ is also smooth. Denote now by $M\subset \mathbb{R}^d$ the compact subset of $\mathbb{R}^d$ whose boundary is given by the smooth hypersurface $\Gamma=\partial M$. With this notation we have:

\begin{lemma}
\label{lem:image}
Let $L=N_1^+\partial M$ be the unit normal bundle of a smooth closed hypersurface $\partial M \subset \mathbb{R}^d$. Then its image $L_{f,\varepsilon}=\beta_{f,\varepsilon}(N_1^+\partial M)$ by the boundary map, is a smooth Legendrian submanifold whose front is the outer normal $\varepsilon$-equidistant hypersurface of $f(\partial M)$:
\[\pi (L_{f,\varepsilon})=f(\partial M)_{\varepsilon}.\]
In particular, if the front $f(\partial M)_{\varepsilon}$ is also smooth, then the following equality holds:
\[\beta_{f,\varepsilon}(N_1^+\partial M)=N_1^+\partial (F_{f,\varepsilon}(M)).\]
\end{lemma}

\begin{proof}
We use the fact that the boundary map is expressed as a composition:
    \[\beta_{f,\varepsilon}=\phi_{\varepsilon}\circ h_f,\]
    and we show first:
       \[h_f(N_1^+\partial M)=N_1^+f(\partial M),\] 
    or equivalently, for all $(x,n)\in N_{1,x}^+\partial M$, $\big (f(x),\frac{f'(x)^{-T}n}{\|f'(x)^{-T}n\|}\big )\in N_{1,f(x)}^+f(\partial M)$.
 Clearly it suffices to show that for all $x\in \partial M$, $f'(x)^{-T}n\in N_{f(x)}^+f(\partial M)$. To show this, let $n\in N_x^+\partial M$ and notice that for any tangent vector $u\in T_x\partial M$ with $\langle u,n\rangle=0$ (i.e.~orthogonal to $n$):
    \[\langle f'(x)^Tu,f'(x)^{-T}n\rangle=\langle u,n\rangle =0.\]
 Since $f'(x)^Tu\in T_{f(x)}f(\partial M)$ we obtain that $f'(x)^{-T}n\in N_{f(x)}f(\partial M)$. Moreover, as one may easily verify, $f'(x)^{-T}n \in N^+_{f(x)}f(\partial M)$ is an outer unit normal vector, from which we obtain
 \[\beta_{f,\varepsilon}(N_1^+\partial M)=\phi_{\varepsilon}(N_1^+f(\partial M))=L_{f,\varepsilon}.\]
 Now, by definition of the exponential map we have 
 \[\pi(L_{f,\varepsilon})=\pi \circ \beta_{f,\varepsilon}(N_1^+\partial M)=\pi \circ \phi_{\varepsilon}(N_1^+f(\partial M))=\operatorname{Exp}_{\varepsilon}(N_1^+f(\partial M))=f(\partial M)_{\varepsilon},\]
 where $f(\partial M)_{\varepsilon}$ is the $\varepsilon$-equidistant hypersurface of $f(\partial M)$. To finish the proof, notice that if $f(\partial M)_{\varepsilon}$ is smooth, then the hypersurface $\partial_+ \overline{B_{\varepsilon}(f(\partial M))}$ is also smooth, and this gives the equality
 \[\partial_+ \overline{B_{\varepsilon}(f(\partial M))}=f(\partial M)_{\varepsilon}.\]
 Thus, by Lemma \ref{lem:front}
\[N_1^+\partial_+ \overline{B_{\varepsilon}(f(\partial M))}=L_{f,\varepsilon},\] 
and by the definition of the outer boundary
\[\partial_+ \overline{B_{\varepsilon}(f(\partial M))}=\partial \overline{B_{\varepsilon}(f(M))}=\partial (F_{f,\varepsilon}(M)),\]
which proves the Lemma.
\end{proof}

Now we are ready to prove the main Theorem \ref{theorem:invariant0} stated in the introduction, which can be restated as follows:

\begin{theorem}
\label{theorem:invariant}
Let $M\subset \mathbb{R}^d$ be a compact subset with smooth boundary $\partial M$. Then its unit normal bundle $N_1^+\partial M$ is invariant for the boundary map $\beta_{f,\varepsilon}$, if and only if the set $M$ is invariant for the set-valued map $F_{f,\varepsilon}$:
\[\beta_{f,\varepsilon}(N_1^+\partial M)=N_1^+\partial M\Longleftrightarrow F_{f,\varepsilon}(M)=M.\]
\end{theorem}

\begin{proof}
($\Longrightarrow$) Suppose that $N_1^+\partial M$ is invariant for $\beta_{f,\varepsilon}$. Projecting equation 
\[\beta_{f,\varepsilon}(N_1^+\partial M)=N_1^+\partial M,\]
to $\mathbb{R}^d$ gives
\[\pi \circ \beta_{f,\varepsilon}(N_1^+\partial M)=\partial M,\]
which, by Lemma \ref{lem:image}, implies:
\[\partial (F_{f,\varepsilon}(M))=\partial M\]
proving invariance of $M$ due to compactness of $M$:
\[F_{f,\varepsilon}(M)=M.\]
 
 \medskip
 
\noindent ($\Longleftarrow$) Conversely, if $M$ is invariant for $F_{f,\varepsilon}$, then
\[\partial F_{f,\varepsilon}(M)=\partial M,\]
is a smooth hypersurface and thus by Lemma \ref{lem:image}:
\[\beta_{f,\varepsilon}(N_1^+\partial M)=N_1^+\partial F_{f,\varepsilon}(M)=N_1^+\partial M.\]
\end{proof}

It is obvious from the above that if a unit normal bundle $N_1^+\partial M$ of the boundary of a set $M$ is invariant for the boundary map $\beta_{f,\varepsilon}$, then it is also invariant for its inverse $\beta_{f,\varepsilon}^{-1}$, i.e.~$\beta_{f,\varepsilon}^{-1}(N_1^+\partial M)=N_1^+\partial M$. Below we will show that the inverse of the boundary map is itself a boundary map, associated now to the dual (or lower inverse) $F^*_{f,\varepsilon}$ of the set-valued map $F_{f,\varepsilon}$, which is defined as follows:  
\begin{align*}
F_{f,\varepsilon}^*(y) & :=\big \{x\in \mathbb{R}^d: y\in F_{f,\varepsilon}(x)\big \} \\
       & = \big \{x\in \mathbb{R}^d: f(x) \in \overline{B_{\varepsilon}(y)}\big \} \\
       & = f^{-1}\big (\overline{B_{\varepsilon}(y)}\big ) := \bigcup_{z \in \overline{B_{\varepsilon}(y)}}f^{-1}(z) .
\end{align*}
Since $\mathbb{R}^d$ is not compact, it is convenient to extend this dual map $F^*_{f,\varepsilon}$ as a set-valued map taking values in the larger space $\mathcal{C}(\mathbb{R}^d)$ of all closed (but not necessarily bounded) subsets of $\mathbb{R}^d$. 
Then, a closed subset $M^*\subset \mathbb{R}^d$ will be called invariant for the dual map $F^*_{f,\varepsilon}$, if
\[F^*_{f,\varepsilon}(M^*)=M^*.\]
Recall from \cite[Proposition 4.1]{LRR} that there is a close relationship between the invariant sets $M^*$ of the dual map $F_{f,\varepsilon}^*$ and the invariant sets $M$ of $F_{f,\varepsilon}$, generalising the classical decomposition of the state space to attractor-repeller pairs. But in case where the boundary $\partial M$ of the invariant set $M$ is smooth, this decomposition admits the following refinement.  

\begin{prop}
\label{prop:invariant*}
Let $M\subset \mathbb{R}^d$ be a compact subset with smooth boundary $\partial M$. Then $M$ is invariant for $F_{f,\varepsilon}$ if and only if the closure of its complement $M^*:=\overline{\mathbb{R}^d\setminus M}$ is invariant for $F_{f,\varepsilon}^*$.
\end{prop}

\begin{proof}
($\Longrightarrow$) Since $M$ is invariant for $F_{f,\varepsilon}$, it follows from Theorem \ref{theorem:invariant} above that the unit normal bundle $N_1^+\partial M$ of its boundary is invariant for the boundary map $\beta_{f,\varepsilon}$:
\[\beta_{f,\varepsilon}(N_1^+\partial M)=N_1^+\partial M,\]
and thus for its inverse $\beta_{f,\varepsilon}^{-1}$ as well:
\[\beta_{f,\varepsilon}^{-1}(N_1^+\partial M)=N_1^+\partial M.\]
But $\beta_{f,\varepsilon}^{-1}=h_f^{-1}\circ \phi^{-1}_{\varepsilon}$ where
\[\phi_{\varepsilon}^{-1}(x,n)=\phi_{-\varepsilon}(x,n)=(x-\varepsilon n,n)\]
is the time-$(-\varepsilon)$ map of the geodesic flow of $\mathbb{R}^d$, and
\[h^{-1}_f(x,n)=h_{f^{-1}}(x,n)=\big (f^{-1}(x),\frac{f'(f^{-1}(x))^Tn}{\|f'(f^{-1}(x))^Tn\|}\big )\]
is the lift of the inverse diffeomorphism $f^{-1}$ on the unit tangent bundle. In particular
\[\beta_{f,\varepsilon}^{-1}(x,n)=(f^{-1}(x-\varepsilon n),\frac{f'(f^{-1}(x-\varepsilon n))^Tn}{||f'(f^{-1}(x-\varepsilon n))^Tn||}).\]
Let us analyse now the equality
\[\beta_{f,\varepsilon}^{-1}(N_1^+\partial M)=h_f^{-1}\circ \phi_{\varepsilon}^{-1}(N_1^+\partial M)=N_1^+\partial M.\]
To analyse first the image $\phi^{-1}_{\varepsilon}(N_1^+\partial M)$ we denote by:
\[\partial M_{-\varepsilon}:=\operatorname{Exp}_{-\varepsilon}(N_1^+\partial M)=\big \{y\in \mathbb{R}^d: \quad y=x-\varepsilon n, \quad (x,n)\in N_1^+\partial M\big \}.\]
the normal $(-\varepsilon)$-equidistant hypersurface of $\partial M$ (or, what is equivalent, the inner normal $\varepsilon$-equidistant of $\partial M$). We check that this is also a smooth hypersurface; indeed, since $f$ is a diffeomorphism we obtain by projection:
\[\pi \circ \beta_{f,\varepsilon}^{-1}(N_1^+\partial M)=\pi \circ h_f^{-1}\circ \phi_{-\varepsilon}(N_1^+\partial M)=f^{-1}(\pi \circ \phi_{-\varepsilon}(N_1^+\partial M))=f^{-1}(\operatorname{Exp}_{-\varepsilon}(N_1^+\partial M))=f^{-1}(\partial M_{-\varepsilon})=\partial M,\]
and smoothness of $\partial M_{-\varepsilon}$ follows by the smoothness of $\partial M$. From Lemma \ref{lem:front} we obtain:
\[\phi^{-1}_{\varepsilon}(N_1^+\partial M)=N_1^+\partial M_{-\varepsilon},\]
 while another easy calculation gives the second equality below:
\[\beta_{f,\varepsilon}^{-1}(N_1^+\partial M)=h_f^{-1}(N_1^+\partial M_{-\varepsilon})=N_1^+f^{-1}(\partial M_{-\varepsilon})=N_1^+\partial M.\]
Now we use the following two properties:
\[\partial M=\partial M^*\]
and
\[\partial \overline{B_{\varepsilon}(M^*)}\subseteq \partial M_{-\varepsilon},\]
with equality for $\partial M_{-\varepsilon}$ smooth. This implies the following equivalences:
\[f^{-1}( \partial M_{-\varepsilon})=\partial M\Longleftrightarrow f^{-1}(\partial \overline{B_{\varepsilon}(M^*)})=\partial M^*\]
\[\Longleftrightarrow \partial f^{-1}(\overline{B_{\varepsilon}(M^*)})= \partial M^* \Longleftrightarrow \partial F_{f,\varepsilon}^*(M^*)=\partial M^*\Longleftrightarrow F_{f,\varepsilon}^*(M^*)=M^*\]
which is what we wanted to prove.

\medskip

\noindent ($\Longleftarrow$) We work as above but in the backwards direction. Starting from equations 
\[F_{f,\varepsilon}^*(M^*)=M^* \Longleftrightarrow \partial F_{f,\varepsilon}^*(M^*)=\partial M^*\Longleftrightarrow \partial f^{-1}(\overline{B_{\varepsilon}(M^*)})=\partial M^*\Longleftrightarrow f^{-1}(\partial \overline{B_{\varepsilon}(M^*)})=\partial M^*,\]
we obtain by the smoothness of $\partial \overline{B_{\varepsilon}(M^*)}=f(\partial M^*)$ that
\[\partial \overline{B_{\varepsilon}(M^*)}=\partial M_{-\varepsilon}\]
and the equality $\partial M=\partial M^*$ implies:
\[f^{-1}(\partial M_{-\varepsilon})=\partial M.\]
Thus
\[\beta_{f,\varepsilon}^{-1}(N_1^+\partial M)=N_1^+f^{-1}(\partial M_{-\varepsilon})=N_1^+\partial M,\]
which proves $\beta_{f,\varepsilon}^{-1}$-invariance of $N_1^+\partial M$. But this means that $N_1^+\partial M$ is also invariant for $\beta_{f,\varepsilon}$:
\[\beta_{f,\varepsilon}(N_1^+\partial M)=N_1^+\partial M,\]
and thus, by Theorem \ref{theorem:invariant} again, $M$ is invariant for $F_{f,\varepsilon}$ as well. This finishes the proof.
\end{proof}

\section{Normal Hyperbolicity and Persistence of Minimal Invariant Sets}
\label{sec:3}

This section is devoted to the proof of the main persistence Theorem \ref{thm:main2}, which can be restated as follows:

\begin{theorem}\label{thm:persistence}
Let $f$ be a $C^2$-diffeomorphism of $\mathbb{R}^d$, $d\geq 2$, and $\varepsilon>0$. Let $M\subset \mathbb{R}^d$ be a minimal invariant set for the set-valued map $F_{f,\varepsilon}$, and suppose that its boundary $\partial M$ is $C^2$-smooth. Suppose also that the corresponding boundary map $\beta_{f,\varepsilon}$ is normally hyperbolic at the unit normal bundle $N_1^+\partial M$ of the boundary. Then 
\begin{itemize}
\item[(a)] the invariant set $M$ is a minimal attractor, 
\item[(b)] for any diffeomorphism $\tilde{f}$ $C^2$-close to $f$ and any $\tilde{\varepsilon}$ sufficiently close to $\varepsilon$, there exists a unique invariant set $\tilde{M}\subset \mathbb{R}^d$ for the set-valued map $F_{\tilde{f},\tilde{\varepsilon}}$, whose boundary $\partial \tilde{M}$ is diffeomorphic and $C^2$-close to $\partial M$, and
\item[(c)] the perturbed invariant set $\tilde{M}$ of $F_{\tilde{f},\tilde{\varepsilon}}$ is also a minimal attractor.
\end{itemize}
\end{theorem}

 We recall (cf.~\cite{HPS}) that a diffeomorphism $\beta:X\rightarrow X$ on a Riemannian manifold $X$ is called normally hyperbolic at a compact submanifold $L\subset X$, if $L$ is invariant, $\beta(L)=L$, and the tangent bundle of $X$ restricted to $L$ admits a continuous Whitney decomposition:
\[T_LX=TL\oplus N^sL\oplus N^uL,\]
into $d\beta$-invariant subspaces, where the rate of contraction (resp. expansion) along the stable (resp. unstable) normal bundles $N^sL$ (resp. $N^uL$) of $L$,  dominates the tangential ones on $TL$. In particular, there exist constants $0<\lambda<\mu^{-1}<1$ and $c>0$ such that:
\[\|d\beta^k(e_s)\|\leq c\lambda^k\|e_s\|, \quad \forall e_s\in N^sL, \quad \forall k\geq 1,\]
\[\|d\beta^{-k}(e_u)\|\leq c\lambda^k\|e_u\|, \quad \forall e_u\in N^uL, \quad \forall k\geq 1,\]
and
\[\|d\beta^k(e_L)\|\leq c\mu^{|k|}\|e_L\|, \quad \forall e_L\in TL, \quad \forall k\in \mathbb{Z},\]
where the corresponding norms are taken with respect to some Riemannian metric on $X$.

In the case of boundary maps $\beta=\beta_{f,\varepsilon}$, and more generally for arbitrary contactomorphisms of the unit tangent bundle (or of any contact manifold), normal hyperbolicity at a Legendrian submanifold is restricted, in the sense that there cannot exist any normal saddles. 

\begin{prop}
\label{prp:attractive-repelling}
Let $\beta:X\rightarrow X$ be a contactomorphism of some $(2d-1)$-dimensional contact manifold $(X,\Delta)$ (where $\Delta \subset TX$ is the contact distribution), and suppose that it is normally hyperbolic at a Legendrian submanifold $L\subset X$. Then it is either normally attracting or normally repelling at $L$.
\end{prop}
\begin{proof}
Normal hyperbolicity implies that there is a continuous splitting of the restriction of the tangent bundle of $X$ at $L$:
\[T_LX=TL\oplus N^uL\oplus N^sL,\]
where $N^uL$ and $N^sL$ are the unstable and stable normal bundles of $L$ respectively. Since $\Delta$ defines a hyperplane in each tangent space of $X$, there always exists at least one normal vector to $L$, either in $N^uL$ or $N^sL$, which is transversal to $\Delta$, say $e_s\in N^sL$, $e_s\pitchfork \Delta$. Now, we claim that 
\[N^uL=0,\]
which implies that $L$ is normally attracting (the proof of the implication $e_u\in N^uL$ with $e_u\pitchfork \Delta \Longrightarrow N^sL=0$ is similar). Indeed, project $N^uL$ onto $\Delta$ along the vector $e_s\in N^sL$, and denote by $\pi_s$ the corresponding projection map. The image $\pi_s(N^uL):=N^u_{\Delta}\subseteq \Delta$ of $N^uL$ under this projection, is of the same dimension with $N^uL$ (by transversality), $\dim N^u_{\Delta}=\dim N^uL$, and for each vector $v\in N^u_{\Delta}$ there exists a unique vector $v_u\in N^uL$ and a vector $v_s\in N^sL$, such that:
\[v-v_u=v_s.\]
Applying $k$-iterates $(d\beta)^k:=d\beta \circ \cdots \circ d\beta$ of the differential of the map $\beta$ to this relation, we obtain, after passing to the limit $k\rightarrow \infty$:
\[\lim_{k\rightarrow \infty}(d\beta)^k(v-v_u)=0,\]
(since $\lim_{k\rightarrow \infty}(d\beta)^k(v_s)=0$). This implies in turn that:
\[\lim_{k\rightarrow \infty}(d\beta)^k(N^u_{\Delta})=N^uL.\]
But since $N^u_{\Delta}\subset \Delta$ and $\Delta$ is invariant by $d\beta$, it follows that $N^uL$ (the limit of iterates of $N_{\Delta}^u$)  will also belong to $\Delta$:
\[N^uL\subset \Delta,\]
i.e.~the strongly unstable normal bundle wholly belongs to the contact hyperplane.
Now, by the fact that $L$ is Legendrian, $TL\subset \Delta$ and we obtain from the above that: 
\[TL\oplus N^uL\subseteq \Delta\]
as well. But by normal hyperbolicity, we know that the (continuous) distribution $TL\oplus N^uL$ admits an integral invariant submanifold $W^u\subset T_1\mathbb{R}^d$ (the unstable manifold) along the points of $L$:
\[T_LW^u=TL\oplus N^uL.\]
But then $T_LW^u\subset \Delta$ as well, i.e.~$W^u$ is also an integral submanifold of the contact structure $\Delta$ over the points of $L$. This can happen in turn only if $W^u$ is a Legendrian submanifold itself, and in particular if $\dim W^u=d-1$. This implies that necessarily $N^uL=0$, and so:
\[T_LX=TL\oplus N^sL,\]
i.e.~$L$ is normally attracting as claimed.    
\end{proof}

In the special case of contactomorphisms given by boundary maps $\beta=\beta_{f,\varepsilon}$ of the unit tangent bundle $X=T_1\mathbb{R}^d$ (endowed with its natural contact structure $\Delta=\Delta_a$), normal hyperbolicity at the unit normal bundle $L=N_1^+\partial M$ of a smooth closed hypersurface $\partial M$ implies by Proposition \ref{prp:attractive-repelling} above, either normal attraction or repulsion at $N_1^+\partial M$. As we will see below, this property is inherited (by projection) to the boundary $\partial M$ of the corresponding invariant set $M$ of the set-valued map $F_{f,\varepsilon}$ in the following sense:

\begin{definition}
    \label{def:attractingboundary}
    Let $M$ be a compact $F_{f,\varepsilon}$-invariant set. We say that the boundary $\partial M$ of $M$ is \emph{attracting} for $F_{f,\varepsilon}$ if there exists an open neighbourhood $U=B_{\eta}(\partial M)$ of the boundary, such that
    \[\lim_{k\rightarrow \infty}d\big (\partial (F^k_{f,\varepsilon}(U)),\partial M\big )=0.\]
    Analogously, we say that the boundary $\partial M$ is repelling for $F_{f,\varepsilon}$ if it is attracting for the dual set-valued map, i.e.:
    \[\lim_{k\rightarrow \infty}d\big (\partial (F^{*k}_{f,\varepsilon}(U)),\partial M\big )=0.\]
\end{definition}

Using this definition we can show:

\begin{lemma}
\label{lem:bshyperbolic}
Let $M$ be a compact subset with $C^2$-smooth boundary $\partial M$ such that the boundary map $\beta_{f,\varepsilon}$ is normally hyperbolic at $N_1^+\partial M$. Then the boundary $\partial M$ of the invariant set $M$ is either attracting (in the normally attracting case) or repelling (in the normally repelling case) for $F_{f,\varepsilon}$, in the sense of Definition \ref{def:attractingboundary}.   
\end{lemma}
\begin{proof}
By Proposition \ref{prp:attractive-repelling} above, since the boundary map $\beta_{f,\varepsilon}$ is a contactomorphism, it will either be normally attracting or normally repelling at the Legendrian submanifold $N_1^+\partial M$. Suppose first that it is normally attracting. According to Definition \ref{def:attractingboundary} we have to show that there exists an open neighbourhood $U=B_{\eta}(\partial M)$ of the boundary such that
\[\lim_{k\rightarrow \infty}d\big (\partial (F^k_{f,\varepsilon}(U)),\partial M\big )=0.\]
Clearly, by choosing $\eta>0$ sufficiently small we can identify (cover) $U$ with a tubular neighbourhood of the boundary $\partial M$ in $\mathbb{R}^d$ ( which we denote by the same symbol $U$). Then it suffices to show that for any compact subset $\tilde{M}$ with $C^2$-smooth boundary $\partial \tilde{M}\subset U$ close to $\partial M$, the following equality holds true: 
\[\lim_{k\rightarrow \infty}d\big (\partial (F^k_{f,\varepsilon}(\tilde{M})\big ),\partial M)=0.\]
But since $\beta_{f,\varepsilon}$ is normally attracting at $N_1^+\partial M$, the projections of the images $\beta_{f,\varepsilon}^k(N_1^+\partial \tilde{M})$ onto $\mathbb{R}^d$ are, for all $k\geq 1$, smooth hypersurfaces $C^2$-close to $\partial M$, from which it follows by repeated application of Lemma \ref{lem:image} that:
\[\beta_{f,\varepsilon}^k(N_1^+\partial \tilde{M})=N_1^+\partial (F^k_{f,\varepsilon}(\tilde{M})).\]
In particular:
\[\lim_{k\rightarrow \infty}d_1\big (\beta_{f,\varepsilon}^k(N_1^+\partial \tilde{M}),N_1^+\partial M\big )=\lim_{k\rightarrow \infty}d_1\big (N_1^+\partial (F^k_{f,\varepsilon}(\tilde{M})),N_1^+\partial M\big )=0,\]
where $d_1$ is the distance function on the unit tangent bundle $T_1\mathbb{R}^d=\mathbb{R}^d\times S^{d-1}$ (induced by the one on $T\mathbb{R}^d=\mathbb{R}^d\times \mathbb{R}^d$). Obviously $d_1\geq d$ (with equality on horizontal sets) and thus:
\[d\big (\partial (F^k_{f,\varepsilon}(\tilde{M}),\partial M)\big )=d\big (\pi \circ \beta_{f,\varepsilon}^k(N_1^+\partial \tilde{M}),\pi(N_1^+\partial M)\big )\leq d_1\big (\beta_{f,\varepsilon}^k(N_1^+\partial \tilde{M}),N_1^+\partial M\big ),\]
which implies that:
\[\lim_{k\rightarrow \infty}d\big (\partial (F^k_{f,\varepsilon}(\tilde{M})),\partial M\big )=\lim_{k\rightarrow \infty}d_1\big (\beta_{f,\varepsilon}^k(N_1^+\partial \tilde{M}),N_1^+\partial M\big )=0.\]
This finishes the proof for the normally attracting case. 

For the normally repelling case we work in the same way as above but for the dual map instead. Briefly, to say that the boundary map $\beta_{f,\varepsilon}$ is normally repelling at $N_1^+\partial M$, is equivalent to say that the inverse boundary map $\beta_{f,\varepsilon}^{-1}$ is normally attracting at $\partial M=\partial M^*$, where $M^*=\overline{\mathbb{R}^d\setminus M}$. Then, as before, since $\beta_{f,\varepsilon}^{-1}$ is normally attracting, for any closed  set $\tilde{M}^*=\overline{\mathbb{R}^d\setminus \tilde{M}}$ with $C^2$-smooth boundary $\partial \tilde{M}^*\subset U$ sufficiently close to $\partial M^*$, the projections of all the iterates $\beta_{f,\varepsilon}^{-k}(N_1^+\partial \tilde{M}^*)$, $k\geq 1$,  will be smooth hypersurfaces $C^2$-close to $\partial M$. Thus, by an analogous argument of Lemma \ref{lem:image} for the inverse of the boundary map and the dual of the set-valued map (see also the proof of Proposition \ref{prop:invariant*}), we obtain for all $k\geq 1$:
\[\beta_{f,\varepsilon}^{-k}(N_1^+\partial \tilde{M}^*)=N_1^+\partial (F^{*k}_{f,\varepsilon}(\tilde{M}^*)),\]
and moreover:
\[\lim_{k\rightarrow \infty}d_1\big (\beta_{f,\varepsilon}^{-k}(N_1^+\partial \tilde{M}^*),N_1^+\partial M^*\big )=\lim_{k\rightarrow \infty}d_1\big (N_1^+\partial (F^{*k}_{f,\varepsilon}(\tilde{M}^*)),N_1^+\partial M^*\big )=0.\]
But since $d_1\geq d$ we obtain:
\[d\big (\partial (F^{*k}_{f,\varepsilon}(\tilde{M}^*),\partial M^*)\big )=d\big (\pi \circ \beta_{f,\varepsilon}^{-k}(N_1^+\partial \tilde{M}^*),\pi(N_1^+\partial M^*)\big )\leq d_1\big (\beta_{f,\varepsilon}^{-k}(N_1^+\partial \tilde{M}^*),N_1^+\partial M^*\big ),\]
and conclude that:
\[\lim_{k\rightarrow \infty}d\big (\partial (F^{*k}_{f,\varepsilon}(\tilde{M}^*),\partial M^*)\big )=\lim_{k\rightarrow \infty}d_1\big (\beta_{f,\varepsilon}^{-k}(N_1^+\partial \tilde{M}^*),N_1^+\partial M^*\big )=0.\]
This finishes the proof for the repelling case as well.
\end{proof}

Now we are ready to prove the main Theorem \ref{thm:persistence}. The proof splits into three parts as in the statement of the theorem: (a) normal hyperbolicity of the boundary map at the unit normal bundle of the boundary of a minimal invariant set implies minimal attraction, (b) existence and uniqueness of an invariant set for the perturbed set-valued map with the required properties, and (c) persistence of minimality.

\begin{proof}[Proof of Theorem \ref{thm:persistence}]


\noindent (a) It suffices to show that if the boundary map $\beta_{f,\varepsilon}$ is normally hyperbolic at the unit normal bundle $N_1^+\partial M$ of the boundary of a minimal invariant set $M$, then it is necessarily normally attracting at $N_1^+\partial M$, for then, by Lemma \ref{lem:bshyperbolic} above, $M$ is indeed a minimal attractor for $F_{f,\varepsilon}$. Arguing by contradiction we will show that if the boundary map $\beta_{f,\varepsilon}$ is normally repelling at $N_1^+\partial M$ then the invariant set $M$ cannot be minimal.  To show this, notice that $\beta_{f,\varepsilon}$ being normally repelling at $N_1^+\partial M$ is equivalent to $\beta_{f,\varepsilon}^{-1}$ being normally attracting at $N_1^+\partial M$  which, by Lemma \ref{lem:bshyperbolic}, is in turn equivalent to the boundary $\partial M=\partial M^*$ being attracting for the dual set-valued map $F^*_{f,\varepsilon}$, where $M^*=\overline{\mathbb{R}^d\setminus M}$. Let
\[\mathcal{A}(M^*):=\big \{x\in \mathbb{R}^d: \lim_{k\rightarrow \infty}d( F^{*k}_{f,\varepsilon}(x),M^*)=0\big \}\]
be the region of attraction of $M^*$. This is a forward invariant set for $F^*_{f,\varepsilon}$, since $x\in \mathcal{A}(M^*)$ implies by definition that
\[\lim_{k\rightarrow \infty}d\big (F^{*k}_{f,\varepsilon}(x),M^*\big )=\lim_{k\rightarrow \infty}d\big (F^{*k-1}_{f,\varepsilon}(F^*_{f,\varepsilon}(x)),M^*\big )=0,\]
which means that $F^*_{f,\varepsilon}(x)\subset \mathcal{A}(M^*)$ as well. Now we will show that the closure of the complement $\hat{M}:=\overline{M\setminus \mathcal{A}(M^*)}$ of the region of attraction of $M^*$, is invariant for $F_{f,\varepsilon}$, i.e.
\[F_{f,\varepsilon}(\hat{M})=\hat{M},\]
which will contradict the minimality of $M$, since $\hat{M}\subsetneq M$ is a compact subset. To do so, it suffices to show that $\hat{M}$ is forward invariant. Suppose on the contrary that there exists $x\in F_{f,\varepsilon}(\hat{M})$ such that $x\notin \hat{M}$, i.e.~$x\in \mathrm{Int}(\mathcal{A}(M^*))$. This means that there exists $y\in \hat{M}$ such that $x\in F_{f,\varepsilon}(y)$ and $x\in \mathrm{Int}(\mathcal{A}(M^*))$. This is in turn equivalent to say that there exists $y\in \hat{M}$ such that
\[x\in \overline{B_{\varepsilon}(f(y))}\Longleftrightarrow f(y)\in \overline{B_{\varepsilon}(x)}\Longleftrightarrow y\in f^{-1}(\overline{B_{\varepsilon}(x)})\Longleftrightarrow y\in F_{f,\varepsilon}^*(x)\]
and $x\in \mathrm{Int}(\mathcal{A}(M^*))$. But this contradicts the fact that $y\in \hat{M}=\overline{M\setminus \mathcal{A}(M^*)}$ since the set $\mathcal{A}(M^*)$ is positively invariant for $F^*_{f,\varepsilon}$. 
Thus we have proved that $\hat{M}$ is indeed forward invariant for $F_{f,\varepsilon}$, which contradicts the minimality of $M$. 
To finish the proof it suffices to show that the minimal invariant set $M$ is also an attractor. By Lemma \ref{lem:bshyperbolic} we know that the boundary $\partial M$ of $M$ is attracting for $F_{f,\varepsilon}$ in the sense of Definition \ref{def:attractingboundary}. This implies that for any $\delta>0$ arbitrarily small, there exists $k\geq 1$ such that for all $s\geq k$:
\[\partial (F^s_{f,\varepsilon}(U))\subset B_{\delta}(\partial M).\]
Since $B_{\delta}(\partial M)\subseteq B_{\delta}(M)$, this also implies
\[\partial (F^s_{f,\varepsilon}(U))\subset  B_{\delta}(M),\]
or equivalently
\[\partial (F^s_{f,\varepsilon}(U))\cap (\mathbb{R}^d\setminus B_{\delta}(M))=\emptyset.\]
Suppose now that the complement $\mathbb{R}^d\setminus B_{\delta}(M)$ consists of a finite number of connected components, say $k\geq 1$. But the condition above implies that it decomposes into a disjoint union of $2k$ relatively open sets:
\[\mathbb{R}^d\setminus B_{\delta}(M)=\big ( (\mathbb{R}^d\setminus B_{\delta}(M))\cap \textrm{Int}(F^s_{f,\varepsilon}(U))\big )\cup \big ( (\mathbb{R}^d\setminus B_{\delta}(M))\cap \textrm{Int}(\mathbb{R}^d\setminus F^s_{f,\varepsilon}(U)) \big ).\]
which is impossible since the complement $\mathbb{R}^d\setminus B_{\delta}(M)$ has only $k$ connected components. Thus,  either of the two components in the above decomposition must be empty. Since $\textrm{Int}(F^s_{f,\varepsilon}(U))$ is bounded, only the first component of the decomposition can be empty, i.e.
\[\textrm{Int}(F^s_{f,\varepsilon}(U))\subset B_{\delta}(M).\] 
Thus
\[F^s_{f,\varepsilon}(U)\subset B_{\delta}(M)\]
as well. But since $U=B_{\eta}(\partial M)\subset V=B_{\eta}(M)$ and $V\setminus U\subset M$, so that $F^s_{f,\varepsilon}(V\setminus U)\subset M\subset B_{\delta}(M)$ (where the first inclusion follows by invariance of $M$), we obtain from the above the existence of an open set $V=B_{\eta}(M)$ such that for any $\delta>0$ there exists $k\geq 1$ with
\[F^s_{f,\varepsilon}(V)\subset B_{\delta}(M),\quad \forall s\geq k,\]
or equivalently
\[\lim_{k\rightarrow \infty}d(F^{k}_{f,\varepsilon}(V),M)=0.\]
Thus the minimal invariant set $M$ is indeed attracting for $F_{f,\varepsilon}$ and this finishes the proof of part (a). 


\medskip

\noindent(b) By part (a) above, we can suppose that the boundary map $\beta_{f,\varepsilon}$ is normally attracting at the unit normal bundle $L=N_1^+\partial M$ of the $C^2$-smooth boundary $\partial M$ of a minimal invariant set $M$ for $F_{f,\varepsilon}$. Clearly, by Theorem \ref{theorem:invariant}, it suffices to show that for any diffeomorphism $\tilde{f}$ $C^2$-close to $f$ and any $\tilde{\varepsilon}$ sufficiently close to $\varepsilon>0$, the corresponding boundary map $\beta_{\tilde{f},\tilde{\varepsilon}}$ admits a unique invariant Legendrian submanifold $\tilde{L}=N_1^+\partial \tilde{M}$, which is the unit normal bundle over a smooth closed hypersurface $\partial \tilde{M}$, diffeomorphic and $C^2$-close to $\partial M$. For the proof of this fact we will use the graph transform naturally associated to the perturbed boundary map $\beta_{\tilde{f},\tilde{\varepsilon}}$, which is defined by the following construction. 

Identify first a neighbourhood of the zero section of the normal bundle $NL\subset T(T_1\mathbb{R}^d)$ of $L=N_1^+\partial M$ with a tubular neighbourhood $U_L\subset T_1\mathbb{R}^d$ of $L$ using an appropriate exponential map $\phi: NL\rightarrow T_1\mathbb{R}^d$. Under this identification, the image of any section $\sigma:L\rightarrow NL$ of the normal bundle of $L$ near the zero section $\{0\}=L$, is mapped to its ``graph" $L_{\sigma}=\phi \circ \sigma(L)\subset U_L$, which is a $(d-1)$-dimensional submanifold $C^1$-close and diffeomorphic to $L$. Let us denote by $\mathcal{S}(L)$ the space of $C^1$-smooth (in fact Lipschitz continuous suffices) sections $\sigma:L\rightarrow NL$ of a neighbourhood of the zero section of the normal bundle of $L$. Then, the graph transform $\mathcal{G}_{\tilde{f},\tilde{\varepsilon}}$ associated to the perturbed boundary map $\beta_{\tilde{f},\tilde{\varepsilon}}$ is defined on the space $\mathcal{S}(L)$ by the rule:
\begin{equation}
\label{eq:gt}
L_{\mathcal{G}_{\tilde{f},\tilde{\varepsilon}}(\sigma)}=\beta_{\tilde{f},\tilde{\varepsilon}}(L_{\sigma}), \quad \sigma \in S(L).
\end{equation}

According to the main theorem of normal hyperbolicity \cite{HPS}, for appropriate bounds in the Lipschitz constant of the sections $\sigma$ and the diameter of the tubular neighbourhood $U_L$, the graph transform $\mathcal{G}_{\tilde{f},\tilde{\varepsilon}}$ induces a contraction $\mathcal{G}_{\tilde{f},\tilde{\varepsilon}}:\mathcal{S}(L)\rightarrow \mathcal{S}(L)$ in the space $\mathcal{S}(L)$. Let us denote by $\tilde{\sigma}\in \mathcal{S}(L)$ its unique fixed point $\mathcal{G}_{\tilde{f},\tilde{\varepsilon}}\tilde{\sigma}=\tilde{\sigma}$, and by $\tilde{L}=L_{\tilde{\sigma}}=\phi\circ \tilde{\sigma}(L)\subset U_L$ its graph. This is the unique $(d-1)$-dimensional manifold invariant by the perturbed boundary map $\beta_{\tilde{f},\tilde{\varepsilon}}(\tilde{L})=\tilde{L}$, which is diffeomorphic and $C^1$-close to $L=N_1^+\partial M$. Thus, to prove part (b) of the theorem, it suffices to show that the manifold $\tilde{L}$ is Legendrian, for then, standard transversality arguments show that it will necessarily be the unit normal bundle $\tilde{L}=N_1^+\partial \tilde{M}$ over a smooth closed hypersurface $\partial \tilde{M}$, diffeomorphic and $C^2$-close to $\partial M$.  

In order to prove the Legendrianity of the invariant manifold $\tilde{L}=L_{\tilde{\sigma}}$ we need first to make sense of what a Legendrian section of the normal bundle $NL$ of $L$ must be. To do so we may consider the standard contact structure of the normal bundle $NL$ of $L$, induced by the natural contact structure $\hat{a}=\theta-dt$ of its $1$-jet bundle $J^1(L)\cong T^*L\times \mathbb{R}$ (where $\theta$ is the canonical Liouville form of $T^*L$) through the canonical isomorphisms: 
\[NL=NL_{\Delta_a}\oplus \Delta_a^{\perp}\cong T^*L\times \mathbb{R}\cong J^1(L),\]
where $NL_{\Delta_a}$ is the normal bundle of $L$ inside the contact distribution $\Delta_a$, and $\Delta_a^{\perp}$ is the normal line bundle to the contact distribution inside $T(T_1\mathbb{R}^d)$. Then, a Legendrian section $\sigma:L\rightarrow NL$ is a section whose image $\sigma(L)\subset NL$ is a Legendrian submanifold with respect to the above contact structure $\Delta_{\hat{a}}$. Denote now by $\mathcal{S}_{\textrm{Leg}}(L)\subset \mathcal{S}(L)$ the subset of all such Legendrian sections of $NL$, which are $C^1$-close to the zero section $\{0\}=L$. According to a contact version of Weinstein's tubular neighbourhood theorem (cf.~\cite{Gei}), it can be shown that the exponential map $\phi:NL\rightarrow T_1\mathbb{R}^d$ described in the construction of the graph transform above, can in fact be made a contactomorphism for the natural contact structures of both spaces. This implies in turn that the graph $L_{\sigma}=\phi \circ \sigma(L)\subset U_L$ of any Legendrian section $\sigma \in \mathcal{S}_{\textrm{Leg}}(L)$ is indeed a Legendrian submanifold for the canonical contact structure $\Delta_{a}$ of $T_1\mathbb{R}^d$, and vice-versa. 

Due to the above identification, the graph transform $\mathcal{G}_{\tilde{f},\tilde{\varepsilon}}$ of the perturbed boundary map $\beta_{\tilde{f},\tilde{\varepsilon}}$ naturally acts on the space $\mathcal{S}_{\textrm{Leg}}(L)$ of Legendrian sections by the same rule (\ref{eq:gt}). Since the boundary map $\beta_{\tilde{f},\tilde{\varepsilon}}$ is a contactomorphism of $T_1\mathbb{R}^d$, it follows again by rule (\ref{eq:gt}) that its graph transform $\mathcal{G}_{\tilde{f},\tilde{\varepsilon}}$ preserves the subspace $\mathcal{S}_{\textrm{Leg}}(L)$ of Legendrian sections, and consequently induces a contraction on it $\mathcal{G}_{\tilde{f},\tilde{\varepsilon}}|_{\mathcal{S}_{\textrm{Leg}}(L)}:\mathcal{S}_{\textrm{Leg}}(L)\rightarrow \mathcal{S}_{\textrm{Leg}}(L)$ (as the restriction of a contraction from the ambient space $\mathcal{S}(L)$). Thus, to prove (b) it suffices to show that the space $\mathcal{S}_{\textrm{Leg}}(L)$ is a closed subspace of the space $\mathcal{S}(L)$, for then, the unique fixed point $\tilde{\sigma}\in \mathcal{S}(L)$ of $\mathcal{G}_{\tilde{f},\tilde{\varepsilon}}$, will necessarily be Legendrian $\tilde{\sigma}\in \mathcal{S}_{\textrm{Leg}}(L)$. To verify now the closedness of $\mathcal{S}_{\textrm{Leg}}(L)$ in $\mathcal{S}(L)$ we may use the canonical contact structure $\hat{a}=\theta-dt$ of  $J^1(L)$, and notice that the condition of being Legendrian in $J^1(L)$ (and thus in $NL$) is given by a continuous closed relation (a partial differential equation of 1st order):
\[(j^1u)^*\hat{a}=0\Longleftrightarrow (du)^*\theta=du, \quad u \in C^1(L),\]
which is obviously preserved under $C^1$-limits $u_n\rightarrow u$ of functions in $C^1(L)$.   

To finish the proof of (b) it remains to show that the invariant Legendrian manifold $\tilde{L}$ thus obtained is indeed the unit normal bundle over a $C^2$-smooth hypersurface $\partial \tilde{M}$, which is diffeomorphic and $C^2$-close to $\partial M$. But this follows in turn by standard transversality arguments; since $\tilde{L}$ is $C^1$-close to $L=N_1^+\partial M$, its tangent space $T\tilde{L}$ will be $C^0$-close to the tangent space $TL$ of  $L$ (in the metric of the corresponding Grassmann bundle of $(d-1)$-dimensional linear subspaces), and consequently it remains transversal to the vertical directions $TS^{n-1}=\ker d\pi$ of the unit tangent bundle $\pi:T_1\mathbb{R}^d\rightarrow \mathbb{R}^d$, in the sense that it makes a non-zero angle with them. It follows from this that the front projection $\pi(\tilde{L})=\Gamma$ of the perturbed Legendrian manifold $\tilde{L}$ is indeed a $C^2$-smooth closed hypersurface, which is $C^2$-close and diffeomorphic to $\partial M$. Thus, $\tilde{\Gamma}=\partial \tilde{M}$ for some compact subset $\tilde{M}\subset \mathbb{R}^d$, and by Lemma \ref{lem:front}, the invariant Legendrian manifold $\tilde{L}$ of $\beta_{\tilde{f},\tilde{\varepsilon}}$, is the unit normal bundle over its front, i.e.~$\tilde{L}=N_1^+\partial \tilde{M}$. This finishes the proof of part (b).

\medskip 
 
 \noindent (c)       
To complete the proof of the theorem we need to show that the invariant set $\tilde{M}$ thus obtained, is also minimal for the perturbed set-valued $F_{\tilde{f},\tilde{\varepsilon}}$ since, by normal hyperbolicity and Lemma \ref{lem:bshyperbolic}, the invariant set $\tilde{M}$ will then necessarily be a minimal attractor. We argue by contradiction and we suppose that $\tilde{M}$ is not minimal. In fact we may assume that there exists a sequence $((f_n, \varepsilon_n))_{n\in\mathbb N}$ with $\lim_{n\rightarrow \infty} \varepsilon_n = \varepsilon$ and $\lim_{n\rightarrow \infty} f_n = f$ (in the $C^2$-sense)  such that the corresponding invariant sets $M_n$ are not minimal, i.e.~there exist smaller invariant sets $\hat M_n\subsetneq M_n$ for each $F_{f_{n},\varepsilon_n}$. By the same arguments as in (a), since the boundary $\partial M$ is attracting, all of the hypersurfaces $\partial M_{n}$ will be close to it in the Hausdorff metric, and thus the boundaries $\partial \hat{M}_n$ of each of the minimal invariant sets $\hat{M}_n$ of $F_{f_n,\varepsilon_n}$, will necessarily be situated far from the boundaries $\partial M_{n}$ (and thus from $\partial M$) of the larger invariant sets $M_n$, for all $n\in \mathbb{N}$. So we can assume that there always exists a point $\xi_n \in (M_n\setminus \hat M_n)\cap \interior{M}$, which is situated far from the boundary $\partial M_n$ (i.e.~such that $d(\xi_n, \partial M_n) \ge \delta$ for all $n\in \mathbb N$ and some $\delta >0$). Since $M$ is compact, we may assume without loss of generality that $\lim_{n\to\infty}\xi_n=\xi \in \interior{M}$. Due to compactness of $M$ again, we can also assume without loss of generality that there exists a $y\in \hat M_n\cap \interior{M}$ for all $n\in\mathbb N$. Consider now the open set-valued map $\interior{F}_{f,\varepsilon}$, defined by $\interior{F}_{f,\varepsilon}(x) = B_{\varepsilon}(f(x))$. Since $M$ is minimal, there exists a $k\in\mathbb N$ and $\zeta>0$ with $B_\zeta(\xi)\subset \interior{F}_{f,\varepsilon}^k(y)$. Now there exists an $N\in\mathbb N$ such that for all $n > N$, we have $B_{\zeta/2}(\xi)\subset \interior{F}_{f_n,\varepsilon_n}^k(y)$. In particular, for some  $n> N$, we have $\xi_n\in \interior{F}_{f_n,\varepsilon_n}^k(y)$, which contradicts the minimality of $\hat M_n$. This finishes the proof of (c) and the proof of the theorem. 
\end{proof}

\end{document}